\documentclass[12pt]{amsart}

\usepackage{latexsym, amssymb, amsmath,ulem,soul,esint}

\usepackage{amsfonts, graphicx}
\usepackage{graphicx,color}
\usepackage{url}
\newcommand{\be}{\begin{equation}}
\newcommand{\ee}{\end{equation}}
\newcommand{\beq}{\begin{eqnarray}}
\newcommand{\eeq}{\end{eqnarray}}

\newtheorem{thm}{Theorem}[section]

\newtheorem{lma}{Lemma}[section]

\theoremstyle{remark}

\numberwithin{equation}{section}
\def\tr{\operatorname{tr}}

\def\cR{\mathcal{R}}

\def\be{\begin{equation}}
\def\ee{\end{equation}}
\def\bee{\begin{equation*}}
\def\eee{\end{equation*}}

\newcommand{\ddb}{\partial \ov{\partial}}

\newcommand{\Ric}{\mathrm{Ric}}

\newcommand{\vp}{\varphi}

\def\K{K\"ahler }

\def\KR{K\"ahler-Ricci }
\def\Ric{\text{\rm Ric}}
\def\Rm{\text{\rm Rm}}

\def\p{\partial}

\def\tr{\operatorname{tr}}

\def\e{\varepsilon}
\def\a{{\alpha}}

\def\ddb{\sqrt{-1}\partial\bar\partial}

\begin{document}

\title{K\"ahler tori with almost non-negative scalar curvature}

\author[J. Chu]{Jianchun Chu}
\address[Jianchun Chu]{School of Mathematical Sciences, Peking University, Yiheyuan Road 5, Beijing, P.R.China, 100871}
\email{jianchunchu@math.pku.edu.cn}

\author[M.-C. Lee]{Man-Chun Lee}
\address[Man-Chun Lee]{Department of Mathematics, The Chinese University of Hong Kong, Shatin, N.T., Hong Kong}
\email{mclee@math.cuhk.edu.hk}


\date{\today}

\begin{abstract}
Motivated by the torus stability problem, in this work we study K\"ahler metrics with almost non-negative scalar curvature on complex torus. We prove that after passing to a subsequence, non-collapsing sequence of K\"ahler metrics with almost non-negative scalar curvature will converge to flat torus weakly.
\end{abstract}


\maketitle

\section{Introduction}

In conformal geometry, one can distinguish manifolds using the Yamabe type \cite{Kobayashi87,Schoen1989}. It is well-known that a compact manifold is of positive Yamabe type if and only if $M$ admits a metric with positive scalar curvature. In the study of positive scalar curvature, one of the celebrated Theorem  was proved by Schoen-Yau \cite{SchoenYau1979,SchoenYau1979-2} for $n\leq7$ and Gromov-Lawson \cite{GromovLawson1980} for general $n$ saying that torus $\mathbb{T}^{n}$ cannot admit metrics with positive scalar curvature. In particular, torus is of vanishing Yamabe type.

Motivated by the torus rigidity, Gromov \cite{Gromov2014} conjectured a stability of metrics with almost non-negative scalar curvature on torus. Namely, if a sequence of metrics $g_i$ has uniform bound on diameter and volume, the scalar curvature is almost non-negative and is non-collapsing in appropriate sense, then after passing to a subsequence it will converge to a flat torus weakly. If the scalar curvature is strengthened to the Ricci curvature, the stability in the Gromov-Hausdorff topology follows from the celebrated work of Colding \cite{Colding1997} using volume lower bound as the non-collapsing condition. When $n=3$, using the Ricci flow method of Hochard \cite{Hochard2016} and Simon-Topping \cite{SimonTopping2017}, a sequence of volume non-collapsed metrics $g_i$ on $\mathbb{T}^3$ with Ricci curvature bounded from below and almost non-negative scalar curvature will converge to a flat torus in the Gromov-Hausdorff topology. On the other hand, if one strengthens the assumption so that $g_i\to g_\infty$ in $C^0$, it was proved by Gromov \cite{Gromov2014} and Bamler \cite{Bamler2016} that $d_{g_\infty}=d_{g_{\mathrm{flat}}}$ if $g_\infty$ is a-priori $C^2$. The $C^2$ assumption was later-on removed by Burkhardt-Guim \cite{Burkhardt2019}. However with only upper bound on diameter and volume, it is possible that the sequence has increasingly many and increasingly thin wells. Hence one cannot expect the Gromov-Hausdorff convergence in general. When $n=3$, Sormani \cite{Sormani2017} proposed using the MinA condition to prevent collapsing. She conjectures the convergence in the volume preserving intrinsic flat sense under the MinA condition. Progresses have been made towards the conjecture of Sormani under various settings, see \cite{Allen2020,AHPPS2019,CAP2020,ChuLee2021} and the references therein. We refer interested readers to the survey article \cite{Sormani2021} of Sormani for detailed exposition. See also \cite{LeeNaberNeumayer2020} for the works on stability in different topologies.

In the above mentioned work, most of the results rely on special choice of gauge. To the best of authors' knowledge, the limiting behaviour is still unclear even if the sequence is non-degenerating in the possible strongest sense, namely $g_i$ are uniformly bi-Lipschitz on $\mathbb{T}^n$, although it is known that the sequence is pre-compact in the Gromov-Hausdorff topology. In complex geometry, bi-Lipschitz \K metrics are usually well-behaved. Motivated by this, in this work we study the problem in the \K case of arbitrary complex dimension. Our main result is as follows:
\begin{thm}\label{Main-Thm}
Let $M^n=\mathbb{T}^{2n}$ be a complex torus and $\omega_h$ be a \K metric on $M$. Suppose $\omega_{i,0}$ is a sequence of \K metrics such that for some $\Lambda>1$ and $p>n$, we have
\begin{enumerate}\setlength{\itemsep}{1.5mm}
    \item[(i)] $\mathrm{Vol}(M,\omega_{i,0})\geq \Lambda^{-1}$;
    \item[(ii)] $\|\tr_{\omega_h}\omega_{i,0}\|_{L^{p}(\omega_{h})} \leq \Lambda$;
    \item[(iii)] $\cR(\omega_{i,0})\geq -i^{-1}$ for all $i\in \mathbb{N}$,
\end{enumerate}
then there is a flat \K metric $\omega_{\infty}$ on $M$ such that after passing to a subsequence, $\omega_{i,0}\to \omega_\infty$ in the sense of current  and
\begin{equation}
\lim_{i\rightarrow+\infty}\|v_{i,0}-1\|_{L^{q/n}(\omega_{\infty})} = 0
\end{equation}
for any $q<p$, where $v_{i,0}=\frac{\omega_{i,0}^{n}}{\omega_{\infty}^{n}}$ denotes the volume function. In particular, we have
\begin{equation}
d\mathrm{vol}_{g_{i,0}}\to d\mathrm{vol}_{g_{\infty}}\quad \text{weakly on }M,
\end{equation}
where $g_{i,0}$ and $g_{\infty}$ denote the associated Riemannian metrics of $\omega_{i,0}$ and $\omega_{\infty}$ respectively. If in addition (ii) holds for $p=+\infty$, then for any $x,y\in M$,
\begin{equation}\label{distance estimate}
\limsup_{i\to +\infty} d_{\omega_{i,0}}(x,y)\leq d_{\omega_\infty}(x,y).
\end{equation}
\end{thm}

The volume assumption is necessary to avoid collapsing while we impose $L^p$ bound to ensure non-expanding. We refer readers to \cite{AllenSormani2020} for many interesting examples along this line. With this, we are able to show that the sequence sub-sequentially converges to some flat metric weakly without assumption on special gauge but the K\"ahlerity. The distance estimate in principle says that potentially shorter path might be built along the convergence but not the other way round under almost non-negative scalar curvature, see \cite{LeeNaberNeumayer2020} for examples in the Riemannian case.
\bigskip

{\it Acknowledgement}:
J. Chu was partially supported by Fundamental Research Funds for the Central Universities (No. 7100603592). The authors would like to thank the referee for the useful comments.

\section{Regularization using Ricci flows}

To prove Theorem~\ref{Main-Thm}, we start with modifying the reference \K form using Calabi-Yau Theorem.
\begin{lma}\label{CY-ref}
Under the assumption of Theorem~\ref{Main-Thm}, there is a sequence $\a_i$ of flat \K  metric inside the same \K class of $\omega_{i,0}$  such that $\a_i=\omega_{i,0}+\ddb u_i$ where
\begin{enumerate}\setlength{\itemsep}{1mm}
\item[(a)] $\displaystyle\sup_M u_i=0$ and $\|u_i\|_{L^\infty}\leq L_1$;
\item[(b)] $L_1^{-1}\omega_h \leq \a_i\leq L_1 \omega_h$ on $M$;
\medskip
\item[(c)] $\log  \frac{\omega_{i,0}^n}{\a_i^n}\geq -L_1 i^{-1/2}$ on $M$
\end{enumerate}
for some $L_{1}>1$ and for all $i\in\mathbb{N}$.
\end{lma}
\begin{proof}
We may assume $\int_M \omega_h^n=1$ by rescaling.  By the solution to Yau's solution to the Calabi conjecture \cite{Yau1978}, there is $u_i\in C^\infty(M)$ with $\sup_M u_i=0$ such that
\begin{equation}
\a_i := \omega_{i,0}+\ddb u_i
\end{equation}
is a K\"ahler Ricci flat metric on $M$. In particular, integrating by parts shows that for all $i\in \mathbb{N}$, we have
\begin{equation}\label{CY-lowerbound}
\mathrm{Vol}(M,\a_i)=\mathrm{Vol}(M,\omega_{i,0})\geq \Lambda^{-1}
\end{equation}
and
\begin{equation}\label{CY-upperbound}
\begin{split}
\int_M  \a_i\wedge  \omega_h^{n-1}&=\int_M  \omega_{i,0}\wedge \omega_h^{n-1}\\
&= \frac{1}{n}\int_M \tr_{\omega_h} \omega_{i,0}\cdot  \omega_h^{n}\\
&\leq \frac{1}{n}\cdot\|\tr_{\omega_h}\omega_{i,0}\|_{L^{p}(\omega_{h})}\\
&\leq \frac{1}{n}\cdot\Lambda.
\end{split}
\end{equation}
Since $M$ is a complex torus, then the Ricci flat metric $\alpha_{i}$ must be flat. Therefore, \eqref{CY-lowerbound} and \eqref{CY-upperbound} imply that
\begin{equation}\label{CY-equiv}
C_1^{-1}\omega_h\leq \a_i\leq C_1\omega_h
\end{equation}
for some $C_1>1$ and for all $i\in \mathbb{N}$. From this, (b) follows immediately.

\medskip

To see (a), we use  \eqref{CY-equiv} to see that
\begin{equation}
\begin{split}
\|\Delta_{\omega_{h}} u_i\|_{L^p(\omega_h)}&=\|\tr_{\omega_h} (\a_i-\omega_{i,0})\|_{L^p(\omega_h)}\\[1mm]
&\leq \|\tr_{\omega_h} \a_i\|_{L^p(\omega_h)}+\|\tr_{\omega_h} \omega_{i,0}\|_{L^p(\omega_h)}\\[1mm]
&\leq C_2
\end{split}
\end{equation}
for some $C_2>0$ and for all $i\in \mathbb{N}$. We claim that
\begin{equation}\label{claim}
\|u\|_{L^{\infty}} \leq C\|\Delta_{\omega_{h}} u\|_{L^p(\omega_h)}
\end{equation}
for some constant $C$ and all $u\in C^\infty(M)$ with $\sup_{M}u=0$. Then (a) follows from \eqref{claim} immediately. To prove \eqref{claim}, we argue by contradiction. Suppose that \eqref{claim} is not true, then for some $u_i\in C^\infty(M)$ with $\sup_M u_i=0$, we have
\begin{equation}
\|u_{i}\|_{L^{\infty}} > i\,\|\Delta_{\omega_{h}} u_i\|_{L^p(\omega_h)}.
\end{equation}
Write $v_{i}=u_{i}/\|u_{i}\|_{L^{\infty}}$. Recalling $\sup_M u_i=0$,
\begin{equation}\label{sup inf Delta v i}
\sup_{M}v_{i} = 0, \quad
\inf_{M}v_{i} = -1, \quad
\|\Delta_{\omega_{h}} v_i\|_{L^p(\omega_h)} < i^{-1}.
\end{equation}
Using $p>n$, $W^{2,p}$ estimate and Sobolev embedding (note that the real dimension of $M$ is $2n$), and passing to a subsequence again, we may assume
\begin{equation}\label{v i v infty convergence}
v_{i} \rightarrow v_{\infty} \quad \text{in the $C^{0}$ sense},
\end{equation}
where $v_{\infty}$ is a weak solution of the Laplacian equation. The standard elliptic theory shows that $v_{\infty}$ is smooth. By the strong maximum principle, $v_{\infty}$ is a constant function. This contradicts to \eqref{sup inf Delta v i} and \eqref{v i v infty convergence}.

\medskip

It remains to establish the lower bound of the volume form, i.e. (c). Following the argument of \cite[Theorem 1.9]{HeZeng2019}, we consider the function
\begin{equation}
F_i=\log \frac{\omega_{i,0}^n}{\a_i^n}-\delta u_i,
\end{equation}
where $\delta$ is a constant to be determined later. At its minimum $x_0\in M$, we have
\begin{equation}
\begin{split}
0&\leq \Delta_{\omega_{i,0}} F_i\\[1mm]
&=-\cR(\omega_{i,0})-\delta\left(\tr_{\omega_{i,0}}\a_i-n \right)\\[1mm]
&\leq (i^{-1}+n\delta)-\delta\tr_{\omega_{i,0}}\a_i.
\end{split}
\end{equation}
Combining this with AM-GM inequality,
\begin{equation}
\left(\frac{\a_i^n}{\omega_{i,0}^n}\right)^{1/n}\leq \frac1n \tr_{\omega_{i,0}}\a_i \leq 1+\frac{1}{\delta n i} \quad
\text{at $x_0$}.
\end{equation}
Hence for all $x\in M$,
\begin{equation}
\begin{split}
e^{F_i(x)} &\geq e^{F_i(x_0)} \\[1mm]
&=\frac{\omega_{i,0}^n}{\a_i^n}\cdot e^{-\delta u_i}\Big|_{x_0}\\
&\geq \left( 1+\frac{1}{\delta n i} \right)^{-n}\cdot e^{-\delta\|u_{i}\|_{L^{\infty}}}.
\end{split}
\end{equation}
By taking $\delta=i^{-1/2}$ and using the definition of $F_i$,
\begin{equation}
\log \frac{\omega_{i,0}^n}{\a_i^n} \geq -n\log\left(1+\frac{1}{ni^{1/2}}\right)-2i^{-1/2}\|u_{i}\|_{L^{\infty}} \geq -Ci^{-1/2},
\end{equation}
where we have used (a) in the last inequality.
\end{proof}

To obtain the stability, we will make use of the \KR flow to regularize the sequence $\omega_{i,0}$. We let $\omega_i(t)$ be the solution to the \KR flow starting from $\omega_i(0)=\omega_{i,0}$, i.e.,
\begin{equation}
\left\{
\begin{array}{ll}
\frac{\p}{\p t}\omega_{i}(t) = -\Ric(\omega_{i}(t));\\[2mm]
\omega_i(0)=\omega_{i,0}.
\end{array}
\right.
\end{equation}
Since $M$ is a complex torus, then the first Chern class $c_{1}(M)$ is zero. By \cite[Section 1]{TianZhang2006}, it is known that each $\omega_i(t)$ exists for $t\geq0$. We now obtain estimates by re-writing the flow as the parabolic Monge-Amp\`ere equation:
\begin{equation}
\left\{
\begin{array}{ll}
\dot\varphi_i(t)=\displaystyle \log \frac{\omega_i(t)^n}{\a_i^n};\\
\omega_i(t)=\omega_{i,0}+\ddb\varphi_i(t);\\[2mm]
\varphi_i(0)=0,
\end{array}
\right.
\end{equation}
where $\dot\varphi_i=\frac{\p\varphi_i}{\p t}$ and $\a_i$ is the flat metric obtained in Lemma~\ref{CY-ref}.  We begin with the zeroth order estimates of $\varphi_i$.

\begin{lma}\label{zeroth-estimate}
There is $L_2>0$ such that for all $i\in \mathbb{N}$, we have
$$\sup_{M\times [0,+\infty)} |\varphi_i(t)|\leq L_2.$$
\end{lma}
\begin{proof}
To obtain the upper bound of $\varphi_i(t)$, we consider the function $F_i=\varphi_i-u_i-\e t$ for $\e>0$. For any $T>0$, let $(x_{0},t_{0})$ be the maximum point of $F$ on $M\times[0,T]$. By the maximum principle, at $(x_{0},t_{0})$, we have $\ddb \varphi_i\leq \ddb u_i$ and hence
\begin{equation}
\begin{split}
\partial_t F_i&=\dot\varphi_i-\e\\
&= \log \frac{(\omega_{i,0}+\ddb\varphi_i)^n}{\a_i^n}-\e\\
&\leq \log \frac{(\omega_{i,0}+\ddb u_i)^n}{\a_i^n}-\e\\[2mm]
&= -\e,
\end{split}
\end{equation}
where we have used $\alpha_{i}=\omega_{i,0}+\ddb u_i$ in the last equality. This is impossible for $t_0>0$. Then $t_{0}=0$. Combining this with $\varphi_i(0)=0$, we see that
\begin{equation}
\sup_{M\times[0,T]}F_i \leq \sup_{M}F_i(\cdot,0) = -\inf_{M}u_{i}.
\end{equation}
By letting $\e\to 0$ and followed by letting $T\to+\infty$, we obtain the upper bound of $\varphi_i$ using (a) of Lemma~\ref{CY-ref}. The lower bound is similar by applying the minimum principle to the function $G=\varphi_i-u_i+\e t$ for $\e>0$.
\end{proof}

Next, we derive the estimate on the volume form $\dot\varphi_i=\log \frac{\omega_i^n}{\a_i^n}$. This follows from a slight modification of standard argument.
\begin{lma}\label{time-zeroth-estimate}
There is $L_3>0$ such that for all $i\in \mathbb{N}$ and $(x,t)\in M\times [0,+\infty)$, we have
\begin{enumerate}\setlength{\itemsep}{1mm}
\item[(a)] $\dot\varphi_i\geq -L_3i^{-1/2}$;
\item[(b)] $\dot\varphi_i(t)\leq L_3t^{-1}+n$.
\end{enumerate}
\end{lma}
\begin{proof}
We begin with the lower bound of $\dot{\vp}_{i}$ as it is relatively simpler. By differentiating $\dot\varphi_i$ with respect to time $t$, we have
\begin{equation}\label{dot varphi i}
\left(\frac{\partial}{\partial t}-\Delta_{\omega_i(t)}\right) \dot\varphi_i(t)=0
\end{equation}
and hence the minimum principle implies
\begin{equation}
\inf_{M\times[0,+\infty)}\dot\varphi_i(x,t)\geq \inf_M \dot\varphi_i(0)
=\inf_{M}\log\frac{\omega_{i,0}^{n}}{\alpha_{i}^{n}} \geq -L_1i^{-1/2},
\end{equation}
by using (c) of Lemma~\ref{CY-ref}.

\medskip

To obtain the upper bound of $\dot{\vp}_{i}$ for $t>0$, we consider the function
\begin{equation}
F_i=t\dot\varphi_i-\varphi_i+u_i-nt.
\end{equation}
It is clear that
\begin{equation}\label{Delta vp-u}
\Delta_{\omega_{i}(t)}(\vp_{i}-u_{i})
= \tr_{\omega_i(t)}(\omega_{i}(t)-\alpha_{i})
= n-\tr_{\omega_i(t)}\a_i.
\end{equation}
Combining this with \eqref{dot varphi i}, we have
\begin{equation}
\begin{split}
\left(\frac{\partial}{\partial t}-\Delta_{\omega_i(t)}\right) F_i
&=\dot\varphi_i-\dot\varphi_i+\Delta_{\omega_i(t)}\vp_{i}-\Delta_{\omega_i(t)}u_{i}-n\\
&=-\tr_{\omega_i(t)} \a_i\\[2mm]
&\leq 0.
\end{split}
\end{equation}
Then the required estimate follows from the maximum principle, (a) of Lemma~\ref{CY-ref} and Lemma \ref{zeroth-estimate}.
\end{proof}

We will show that $\omega_i(t)$ is compact in the $C^\infty_{\mathrm{loc}}$ topology. To do this, we need to obtain an upper and lower bounds of $\omega_{i}(t)$ for $t>0$.

\begin{lma}\label{metric-estimate}
There is $L_4>0$ such that for all $i\in \mathbb{N}$ and $(x,t)\in M\times (0,2]$, we have
\begin{equation}\label{tr alpha omega i}
\tr_{\a_i}\omega_i(t)\leq L_4 t^{1-n} e^{\dot\varphi_i(t)}.
\end{equation}
In particular, $\omega_i(t)$ is uniformly equivalent to $\omega_h$ for $t\in(0,2]$, i.e.,
\begin{equation}\label{omega h omega i}
e^{-C_{4}t^{-1}}\cdot\omega_{h} \leq \omega_i(t) \leq e^{C_{4}t^{-1}}\cdot\omega_{h}
\end{equation}
for some $C_{4}>0$ and for all $i\in\mathbb{N}$.
\end{lma}
\begin{proof}
We first show how to use \eqref{tr alpha omega i} to deduce \eqref{omega h omega i}. By (b) of Lemma \ref{time-zeroth-estimate},
\begin{equation}\label{trace alpha omega}
\tr_{\alpha_{i}}\omega_{i} \leq L_{4}t^{1-n}e^{L_{3}t^{-1}+n} \leq e^{Ct^{-1}}.
\end{equation}
Combining this with the following elementary inequality
\begin{equation}
\tr_{\omega_{i}}\a_i \leq \left(\frac{\a_{i}^{n}}{\omega_{i}^{n}}\right)\cdot(\tr_{\a_i}\omega_i)^{n-1}
= e^{-\dot{\vp}_{i}}\cdot(\tr_{\a_i}\omega_i)^{n-1}
\end{equation}
and (a) of Lemma \ref{time-zeroth-estimate}, we have
\begin{equation}\label{trace omega alpha}
\tr_{\omega_{i}}\a_i \leq e^{L_{3}+(n-1)Ct^{-1}}.
\end{equation}
Then \eqref{omega h omega i} follows from  \eqref{trace alpha omega}, \eqref{trace omega alpha} and (b) of Lemma \ref{CY-ref}.

\medskip

It suffices to prove \eqref{tr alpha omega i}. When $n=1$, \eqref{tr alpha omega i} is trivial as $\tr_{\a_i}\omega_i=e^{\dot{\vp}}$. We may assume $n\geq2$. We consider the following function
\begin{equation}\label{test1}
F_{i}(t) = \log \tr_{\a_i}\omega_i(t)+(n-1)\log t+\Lambda\cdot\left(\varphi_i(t)-u_i \right)-\dot\varphi_i(t),
\end{equation}
where $\Lambda$ is a constant to be determined later. 
Let $(x_{0},t_{0})$ be the maximum point of $F_i$ on $M\times[0,2]$. Since $F_i\to -\infty$ as $t\to 0$, we have $t_{0}>0$. At $(x_{0},t_{0})$, by the standard parabolic Schwarz's lemma (see e.g. \cite[Proposition 3.2.4]{SongWeinkove2013}) and the fact that $\a_i$ is flat, we have
\begin{equation}
\left(\frac{\partial}{\partial t}-\Delta_{\omega_i(t_{0})}\right) \log \tr_{\a_i}\omega_i(t_{0}) \leq 0.
\end{equation}
Combining this with \eqref{dot varphi i} and \eqref{Delta vp-u},
\begin{equation}
\left(\frac{\partial}{\partial t}-\Delta_{\omega_i(t_{0})}\right) F_i(t_{0})\leq (n-1) t_{0}^{-1} +\Lambda\cdot\left(\dot\vp_i(t_{0})+n-\tr_{\omega_i(t_{0})}\a\right).
\end{equation}
Applying the maximum principle at $(x_0,t_0)$,
\begin{equation}
\tr_{\omega_i(t_{0})}\a_i \leq (n-1)\Lambda^{-1}t_{0}^{-1}+\dot\vp_i(t_{0})+n.
\end{equation}
By (b) of Lemma~\ref{time-zeroth-estimate}, at $(x_0,t_0)$ we have
\begin{equation}
\begin{split}
\tr_{\a_i}\omega_i(t_{0})  &\leq \left(\frac{\omega_i(t_{0})^n}{\a_i^n}\right) \cdot (\tr_{\omega_i(t_0)} \a_i)^{n-1}\\[1mm]
&\leq e^{\dot\varphi_i(t_{0})} \cdot \left(((n-1)\Lambda^{-1}+L_{3})t_{0}^{-1}+2n\right)^{n-1},
\end{split}
\end{equation}
which implies
\begin{equation}
\log\tr_{\a_i}\omega_i(t_{0}) \leq \dot\varphi_i(t_{0})-(n-1)\log t_{0}+C(n,L_3,\Lambda).
\end{equation}
Choosing $\Lambda=1$ and using the definition of $F$, we conclude that
\begin{equation}
\sup_{M\times [0,2]} F_i = F_i(x_{0},t_{0}) \leq C(n,L_3)+\sup_{M\times [0,2]}|\vp_i-u_i|.
\end{equation}
Combining this with (a) of Lemma \ref{CY-ref} and Lemma \ref{zeroth-estimate}, we have
\begin{equation}
\sup_{M\times[0,2]}\left(e^{-\dot\varphi_i}t^{n-1}\tr_{\a_i}\omega_i\right) \leq C(n,L_1,L_2,L_3).
\end{equation}
\end{proof}

\section{Proof of Main Theorem}

Now we are ready to prove Theorem~\ref{Main-Thm}.
\begin{proof}[Proof of Theorem~\ref{Main-Thm}]
For the smooth solution $\omega_{i}(t)$ of \KR flow , it is well-known that (see e.g. \cite[(3.56)]{SongWeinkove2013})
\begin{equation}
\left(\frac{\partial}{\partial t}-\Delta_{\omega_i(t)}\right) \cR(\omega_i(t)) \geq 0.
\end{equation}
Then $\cR(\omega_i(0))=\cR(\omega_{i,0})\geq-i^{-1}$ and the maximum principle imply
\begin{equation}\label{Scalar-bound}
\cR(\omega_i(t))\geq -i^{-1}
\end{equation}
for all $i\in \mathbb{N}$ and $(x,t)\in M\times [0,+\infty)$. By Lemma~\ref{metric-estimate} and \cite[Theorem 1]{ShermanWeinkove2012}, $\omega_i(t)$ is bounded in $C^k_{\mathrm{loc}}$ for all $k\in \mathbb{N}$ for $t\in(0,2]$. After passing to subsequence, we may assume that $\omega_i(t)\to \omega_\infty(t)$ in $C^\infty_{\mathrm{loc}}(M\times(0,2])$ as $i\to +\infty$.  Moreover, for any $t\in(0,2]$, $\cR(\omega_\infty(t))\geq 0$ by \eqref{Scalar-bound}.

Let $\hat{\omega}$ be a K\"ahler Ricci flat metric on $M$. Since $\omega_\infty(t)$ is K\"ahler, the non-negativity of scalar curvature implies that
\begin{equation}
\Delta_{\omega_{\infty}(t)}\left(\log\frac{\omega_\infty(t)^n}{\hat \omega^n}\right)
= -\cR(\omega_\infty(t)) \leq 0.
\end{equation}
By the strong maximum principle, the function $\displaystyle\log \frac{\omega_\infty(t)^n}{\tilde \omega^n}$ is constant, and then $\omega_\infty(t)$ is Ricci flat. As $M$ is a complex torus, $\omega_\infty(t)$ must be flat. Moreover, as $\omega_\infty(t)$ is a smooth solution to the \KR flow, uniqueness implies that $\omega_\infty(t)$ is independent of time $t$, i.e., $\omega_\infty(t)\equiv \tilde\omega_{\infty}$ for $t\in(0,2]$ for some flat \K metric $\tilde\omega_{\infty}$ on $M$.

It remains to show that $\omega_{i,0}$ converges to $\tilde\omega_{\infty}$ weakly. We first show the convergence in current sense. Since $\omega_{i,0}$ and $\a_i$ are in the same \K class, by Lemma~\ref{CY-ref} and the Banach-Alaoglu Theorem \cite[Chapter III Proposition 1.23]{Demailly}, we may assume that $\omega_{i,0}$ converge weakly to a \K current $\omega_\infty$. It suffices to show that $\omega_\infty\equiv \tilde\omega_\infty$ as a current.

Let $\eta$ be an arbitrary $(n-1,n-1)$ test form on $M$. Using integration by parts,
\begin{equation}\label{n-1 test}
\begin{split}
\int_M \eta \wedge  \omega_{i}(1)&=\int_M \eta \wedge\left(\omega_{i,0}+\ddb \varphi_i(1) \right)\\
&=\int_M \eta \wedge \omega_{i,0}+\int_M \varphi_i(1)\cdot \ddb \eta \\
&=\int_M \eta \wedge \omega_{i,0}+\mathbf{E}.
\end{split}
\end{equation}
Since $\eta$ is a fixed $(n-1,n-1)$ form on $M$, there is $C_{\eta}>0$ such that
\begin{equation}
|\ddb \eta|\leq C_{\eta}\cdot\omega_h^{n}.
\end{equation}
We will show that $\mathbf{E}\to 0$ as $i\to +\infty$. Using $\varphi_i(0)=0$, (b) of Lemma~\ref{CY-ref} and almost non-negativity of $\dot\varphi_i$ from (a) of Lemma~\ref{time-zeroth-estimate}, we have
\begin{equation}
\begin{split}
|\mathbf{E}|
&=\left|\int^1_0 \left(\int_M \dot\varphi_i(t) \cdot \ddb \eta \right) dt\right|\\
&\leq C \int^1_0 \left(\int_M |\dot\varphi_i(t)|  \a_i^n \right)dt\\
&\leq C \int^1_0 \left(\int_M \dot\varphi_i(t)\, \a_i^n\right)  dt+Ci^{-1/2}.
\end{split}
\end{equation}
Using the elementary inequality $x\leq e^x-1$ for all $x\in \mathbb{R}$, $\dot\varphi_i=\log \frac{\omega_i^n}{\a_i^n}$ and the fact that $\omega_i(t)$ and $\a_i$ are in the same class,
\begin{equation}
\begin{split}
|\mathbf{E}|
&\leq C_\eta \int^1_0 \left(\int_M \left(e^{\dot\varphi_i(t)}-1\right) \a_i^n\right)  dt+Ci^{-1/2}\\
&= C_\eta \int^1_0 \left(\int_M\omega_{i}(t)^n-\int_M \a_i^n\right) \;  dt+Ci^{-1/2}\\[1mm]
&=Ci^{-1/2}.
\end{split}
\end{equation}
By letting $i\to +\infty$ in \eqref{n-1 test}, we have
\begin{equation}
\int_M \eta\wedge \tilde\omega_\infty = \int_M \eta \wedge\omega_\infty
\end{equation}
for any test form $\eta$ on $M$. Hence, $\omega_\infty=\tilde\omega_\infty$ and is smooth on $M$.

Thanks to \eqref{Scalar-bound}, we have (see e.g. \cite[(3.65)]{SongWeinkove2013})
\begin{equation}
\frac{\p}{\p t}\left(\log\frac{\omega_{i}(t)^{n}}{\omega_{i,0}^{n}}\right) = -\cR(\omega_i(t)) \leq i^{-1}.
\end{equation}
Then for all $i\in \mathbb{N}$, we have
\begin{equation}\label{Vform-lowerbound}
e^{-1/i}\cdot \omega_i(1)^n \leq \omega_{i,0}^n.
\end{equation}
On the other hand, by the fact that $\omega_i(1)$ and $\omega_{i,0}$ are in the same class,
\begin{equation}\label{same volume}
\int_M \omega_{i}(1)^n=\int_M \omega_{i,0}^n.
\end{equation}
Let $f_{i}=\frac{\omega_{i}(1)^{n}}{\omega_{\infty}^{n}}$. Since $v_{i,0}=\frac{\omega_{i,0}^{n}}{\omega_{\infty}^{n}}$, then \eqref{Vform-lowerbound} and \eqref{same volume} show
\begin{equation}
e^{-1/i}\cdot f_{i} \leq v_{i,0}, \quad
\int_{M}f_{i}\,\omega_{\infty}^{n} = \int_{M}v_{i,0}\,\omega_{\infty}^{n}.
\end{equation}
We compute
\begin{equation}
\begin{split}
\int_{M}|v_{i,0}-f_{i}|\,\omega_{\infty}^{n}
&\leq \int_{M}|v_{i,0}-e^{-1/i}\cdot f_{i}|\,\omega_{\infty}^{n}+(1-e^{-1/i})\int_{M}|f_{i}|\,\omega_{\infty}^{n} \\
&=\int_{M}(v_{i,0}-e^{-1/i}\cdot f_{i})\,\omega_{\infty}^{n}+(1-e^{-1/i})\int_{M}f_{i}\,\omega_{\infty}^{n}\\
&=2(1-e^{-1/i})\int_{M}f_{i}\,\omega_{\infty}^{n}.
\end{split}
\end{equation}
Since $\omega_{i}(1)\rightarrow\omega_{\infty}$ in the $C^{\infty}$ sense, then $f_{i}\rightarrow1$ in the $C^{\infty}$ sense. By letting $i\rightarrow+\infty$, we obtain
\begin{equation}
\lim_{i\rightarrow+\infty}\|v_{i,0}-1\|_{L^{1}(\omega_{\infty})} = 0.
\end{equation}
In particular, $v_{i,0}$ converges to $1$ point-wise almost everywhere on $M$. By assumption (ii) and AM-GM inequality, we obtain
\begin{equation}
\begin{split}
\int_{M}v_{i,0}^{p/n}\,\omega_{\infty}^{n}
&\leq C(n,\omega_{h},\omega_{\infty})\cdot\int_{M}\left(\frac{\omega_{i,0}^{n}}{\omega_{h}^{n}}\right)^{p/n}\,\omega_{h}^{n}\\
&\leq C(n,\omega_{h},\omega_{\infty})\cdot\int_M (\tr_{\omega_h} \omega_{i,0})^p \omega_h^n \\[2mm]
&\leq C(n,\omega_{h},\omega_{\infty})\cdot\Lambda^{p}.
\end{split}
\end{equation}
Then the $L^{q/n}$ convergence follows from the interpolation inequality.

\medskip

If $p=+\infty$ in (ii), then (b) and (c) of Lemma~\ref{CY-ref} imply that $\omega_{i,0}$ are uniformly equivalent to the fixed metric $\omega_h$, i.e.,
\begin{equation}
C_1^{-1}\omega_h \leq \omega_{i,0} \leq C_1\omega_h
\end{equation}
for some $C_{1}>1$ and for all $i\in \mathbb{N}$. The standard argument of \KR flow (see e.g. \cite[Corollary 3.3.5]{SongWeinkove2013}) shows that for all $i\in\mathbb{N}$ and $(x,t)\in M\times[0,2]$,
\begin{equation}
C_2^{-1}\omega_h \leq \omega_{i}(t) \leq C_2\omega_h
\end{equation}
for some $C_{2}>1$. It follows from \cite[Theorem 1.1]{ShermanWeinkove2012} that $|\Rm(\omega_i(t))|\leq Lt^{-1}$ on $M\times(0,2]$ for some $L>0$. By \cite[Corollary 3.3]{SimonTopping2016}, for all $x,y\in M$ and $t\in [0,1]$,
\begin{equation}\label{distance estimate t}
d_{\omega_{i,0}}(x,y)\leq  d_{\omega_{i}(t)}(x,y)+C(n)\sqrt{Lt}.
\end{equation}
Recall that $\omega_i(t)\to \omega_\infty$ in $C^\infty_{\mathrm{loc}}(M\times(0,2])$ as $i\to +\infty$. In \eqref{distance estimate t}, by letting $i\to+\infty$ and followed by letting $t\to0$, we obtain the required distance estimate \eqref{distance estimate}.
\end{proof}

As an application of the method, the following stability is also proved along the same line.
\begin{thm}
Let $(M,\omega_h)$ be a compact \K manifold with $c_1(M)=0$. Suppose $\omega_{i,0}=\a_i+\ddb u_i$ is a sequence of \K metrics such that for some $\Lambda>0$ and $p>n$, we have
\begin{enumerate}\setlength{\itemsep}{1mm}
    \item[(i)] $\Lambda^{-1}\omega_h\leq \a_i\leq \Lambda \omega_h$;
    \item[(ii)] $\|\tr_{\omega_h}\omega_{i,0}\|_{L^{p}(\omega_{h})} \leq \Lambda$;
    \item[(iii)] $\cR(\omega_i)\geq -i^{-1}$ for all $i\in\mathbb{N}$,
\end{enumerate}
then there is a K\"ahler Ricci flat metric $\omega_{\infty}$ on $M$ such that after passing to a subsequence, $ \omega_{i,0}\to \omega_\infty$ in the sense of current  and
\begin{equation}
\lim_{i\rightarrow+\infty}\|v_{i,0}-1\|_{L^{q/n}(\omega_{\infty})} = 0
\end{equation}
for any $q<p$, where $v_{i,0}=\frac{\omega_{i,0}^{n}}{\omega_{\infty}^{n}}$ denotes the volume function. In particular, we have
\begin{equation}
d\mathrm{vol}_{g_{i,0}}\to d\mathrm{vol}_{g_{\infty}}\quad \text{weakly on }M,
\end{equation}
where $g_{i,0}$ and $g_{\infty}$ denote the associated Riemannian metrics of $\omega_{i,0}$ and $\omega_{\infty}$ respectively. If in addition (ii) holds for $p=+\infty$, then for any $x,y\in M$,
\begin{equation}
\limsup_{i\to +\infty} d_{\omega_{i,0}}(x,y)\leq d_{\omega_\infty}(x,y).
\end{equation}
\end{thm}

\begin{proof}[Sketch of Proof]
Since the proof is almost identical to that of Theorem~\ref{Main-Thm}, we only point out the necessary modifications. Since $\a_i$ is uniformly equivalent to $\omega_h$, the proof of Lemma~\ref{CY-ref} shows that we may assume $u_i$ to be bounded uniformly and $\a_i$ is Ricci flat by Yau's Theorem \cite{Yau1978} with uniformly bounded geometry. Then the proof of Lemma~\ref{zeroth-estimate} and Lemma~\ref{time-zeroth-estimate} can be carried over. To obtain the trace estimate, i.e. Lemma~\ref{metric-estimate}, we modify the test function \eqref{test1} with $\Lambda$ chosen to be large depending only on the uniform curvature bound of $\a_i$. With this modification, we have the same estimates as in Lemma~\ref{metric-estimate}. The proof of convergence in Theorem~\ref{Main-Thm} can now be carried over directly.
\end{proof}


\begin{thebibliography}{10}

\bibitem{Allen2020} Allen, B., {\sl Almost non-negative scalar curvature on Riemannian manifolds conformal to tori}, J Geom Anal (2021). https://doi.org/10.1007/s12220-021-00677-2

\bibitem{AHPPS2019} Allen, B.; Hernandez-Vazquez, L.; Parise, D.; Payne, A.; Wang, S., {\sl Warped tori with almost non-negative scalar curvature}, Geom. Dedicata 200 (2019), 153--171.

\bibitem{AllenSormani2020} Allen, B.; Sormani, C., {\sl Relating notions of convergence in geometric analysis}, Nonlinear Anal. 200 (2020), 111993, 33 pp.

\bibitem{Bamler2016} Bamler, R., {\sl A ricci flow proof of a result by gromov on lower bounds for scalar curvature}, Math. Res. Lett. 23 (2016), no. 2, 325--337.


\bibitem{Burkhardt2019} Burkhardt-Guim, P., {\sl Pointwise lower scalar curvature bounds for $C^0$ metrics via regularizing Ricci flow}, Geom. Funct. Anal. 29 (2019), no. 6, 1703--1772.

\bibitem{CAP2020} Cabrera Pacheco, A.-J.; Ketterer, C.; Perales, R., {\sl Stability of graphical tori with almost nonnegative scalar curvature}, Calc. Var. Partial Differential Equations 59 (2020), no. 4, Paper No. 134, 27 pp.

\bibitem{ChuLee2021} Chu, J.; Lee, M.-C., {\sl Conformal tori with almost non-negative scalar curvature}, preprint, arXiv:2103.07003

\bibitem{Colding1997} Colding, T.-H., {\sl Ricci curvature and volume convergence}, Ann. of Math. (2) 145 (1997), no. 3, 477--501.

\bibitem{Demailly} Demailly, J.-P. {\sl Complex analytic and differential geometry}, freely accessible book, (\url{https://www-fourier.ujf-grenoble.fr/~demailly/manuscripts/agbook.pdf}).

\bibitem{Gromov2014} Gromov, M., {\sl Dirac and Plateau billiards in domains with corners}, Cent. Eur. J. Math. 12 (2014), no. 8, 1109--1156.

\bibitem{GromovLawson1980} Gromov, M.; Lawson, H.B., Jr., {\sl Spin and scalar curvature in the presence of a fundamental group. I}, Ann. of Math. (2) 111 (1980), no. 2, 209--230.

\bibitem{HeZeng2019} He, W.; Zeng, Y., {\sl Constant scalar curvature equation and regularity of its weak solution}, Comm. Pure Appl. Math. 72 (2019), no. 2, 422--448.

\bibitem{Hochard2016} Hochard, R., {\sl  Short-time existence of the Ricci flow on complete, non-collapsed 3-manifolds with Ricci curvature bounded from below}, preprint, arXiv:1603.08726

\bibitem{Kobayashi87} Kobayashi, O., {\sl Scalar curvature of a metric with unit volume Math}, Math. Ann. 279 (1987), no. 2, 253--265.

\bibitem{LeeNaberNeumayer2020} Lee, M.-C.; Naber, A.; Neumayer, R., {\sl $d_p$ convergence and $\e$-regularity theorems for entropy and scalar curvature lower bounds}, preprint, arXiv:2010.15663, to appear in Geom. Topol.

\bibitem{Schoen1989} Schoen, R., {\sl Variational theory for the total scalar curvature functional for Riemannian metrics and related topics}, pp. 120--154 in Topics in calculus of variations (Montecatini Terme, 1987), edited by M. Giaquinta, Lecture Notes in Math. 1365, Springer, 1989.

\bibitem{SchoenYau1979} Schoen, R.; Yau, S.-T., {\sl Existence of incompressible minimal surfaces and the topology of three-dimensional manifolds with nonnegative scalar curvature}, Ann. of Math. (2) 110 (1979), no. 1, 127--142.

\bibitem{SchoenYau1979-2} Schoen, R.; Yau, S.-T., {\sl On the structure of manifolds with positive scalar curvature}, Manuscripta Math. 28 (1979), no. 1-3, 159--183.

\bibitem{ShermanWeinkove2012} Sherman, M.; Weinkove, B., {\sl Interior derivative estimates for the K\"ahler-Ricci flow}, Pacific J. Math. 257 (2012), no. 2, 491--501.

\bibitem{SimonTopping2016} Simon, M.; P.-M. Topping., {\sl Local control on the geometry in 3D Ricci flow}, preprint arXiv:1611.06137, to appear in J. Differential Geom.

\bibitem{SimonTopping2017} Simon, M.; P.-M. Topping., {\sl Local mollification of Riemannian metrics using Ricci flow, and Ricci limit spaces}, Geom. Topol. 25 (2021), no. 2, 913--948.

\bibitem{SongWeinkove2013} Song, J.; Weinkove, B., {\sl An introduction to the K\"ahler-Ricci flow}, An introduction to the K\"ahler-Ricci flow, 89--188,
Lecture Notes in Math., 2086, Springer, Cham, 2013.

\bibitem{Sormani2017} Sormani, C., {\sl Scalar curvature and intrinsic flat convergence}, Measure theory in non-smooth spaces, 288--338, Partial Differ. Equ. Meas. Theory, De Gruyter Open, Warsaw, 2017.

\bibitem{Sormani2021} Sormani, C., {\sl Conjectures on Convergence and Scalar Curvature}, preprint, arXiv:2103.10093

\bibitem{TianZhang2006} Tian, G.; Zhang, Z., {\sl On the K\"ahler-Ricci flow on projective manifolds of general type}, Chinese Ann. Math. Ser. B 27 (2006), no. 2, 179--192.

\bibitem{Yau1978}Yau, S.-T., {\sl On the Ricci curvature of a compact K\"ahler manifold and the complex Monge-Amp\`ere equation. I}, Comm. Pure Appl. Math. 31 (1978), no. 3, 339--411.

\end{thebibliography}
\end{document}